\documentclass[11pt]{article}
\usepackage{graphicx} 
\usepackage{bbm}
\usepackage[margin=1in]{geometry}
\usepackage{color}
\usepackage{amsfonts}
\usepackage{algorithm}
\usepackage{algorithmic}
\usepackage{latexsym, amssymb, amsmath, amscd, amsthm, amsxtra}
\usepackage{mathtools}
\usepackage{enumerate}
\usepackage[all]{xy}
\usepackage{mathrsfs}
\usepackage{fancyhdr}
\usepackage{listings}
\usepackage{hyperref}
\usepackage{enumitem}
\usepackage{aligned-overset}

\newcommand{\projH}{\mathcal{P}(\mathcal{H})}

\def\P{\mathbb{P}}

\def\E{\mathbb{E}}

\def\11{\mathbbm{1}}

\def\ER{Erd\H{o}s-R\'enyi\ }

\def\Eb{\mathbb E}

\newcommand{\Pb}{\mathbb P}
\newcommand{\Qb}{\mathbb Q}

\DeclareMathOperator{\Var}{Var}

\newtheorem{thm}{Theorem}[section]

\newtheorem{proposition}[thm]{Proposition}

\newtheorem{lemma}[thm]{Lemma}

\newtheorem{defn}[thm]{Definition}

\numberwithin{equation}{section}



\hypersetup{colorlinks=true,linkcolor=blue}

\title{Detection and Reconstruction of a Random Hypergraph from Noisy Graph Projection}
\author{Shuyang Gong \\ Peking University \and Zhangsong Li \\ Peking University \and Qiheng Xu \\ University of Chicago}
\date{}

\begin{document}

\maketitle

\begin{abstract}
    For a $d$-uniform random hypergraph on $n$ vertices in which hyperedges are included i.i.d.\ so that the average degree in the hypergraph is $n^{\delta+o(1)}$, the projection of such a hypergraph is a graph on the same $n$ vertices where an edge connects two vertices if and only if they belong to a same hyperedge. In this work, we study the inference problem where the observation is a \emph{noisy} version of the graph projection where each edge in the projection is kept with probability $p=n^{-1+\alpha+o(1)}$ and each edge not in the projection is added with probability $q=n^{-1+\beta+o(1)}$. For all constant $d$, we establish sharp thresholds for both detection (distinguishing the noisy projection from an \ER random graph with edge density $q$) and reconstruction (estimating the original hypergraph). Notably, our results reveal a \emph{detection-reconstruction gap} phenomenon in this problem. Our work also answers a problem raised in \cite{BGPY25+}.\footnote{This paper was presented in part at The IEEE International Symposium on Information Theory (ISIT 2026). This version includes complete mathematical proofs and detailed discussions.}
\end{abstract}

\section{Introduction}

Community detection aims to partition a network’s vertices into clusters based on observed noisy associations between them. The most widely studied formalization assumes pairwise interactions: in friendship graphs, social networks, or co-purchase data, edges appear more frequently between vertices of the same latent community \cite{New06}. However, many real-world systems are fundamentally characterized by higher-order interactions. In scientific co-authorship networks, for instance, relationships are intrinsically multi-way, involving groups of researchers rather than isolated pairs \cite{New04}. Such data are naturally modeled by hypergraphs, a perspective that has led to the hypergraph stochastic block model (HSBM), a framework receiving considerable recent attention \cite{GD14, ACKZ15, SZ22, GJ23, DW23, GP23, GP24}.

This work focuses on a particular subproblem within this broad area: recovering multi-vertex association data (i.e., a hypergraph) from the pairwise association data (i.e., a graph). A motivating example is a large team collaborating on multiple projects, where the ground truth multi-way association is defined by project membership. However, due to practical limitations, we may only observe email exchanges between pairs of coworkers \cite{KY04}, which can be simplified to a graph where an edge connects any two individuals who communicated. Typically, such data is represented by constructing a graph from pairwise communications, effectively projecting the richer hypergraph onto a pairwise structure. Since such a projection generally loses information, a key question arises: to what extent can we reconstruct the original hypergraph from its (potentially noisy) graph projection, under the assumption that vertices within the same hyperedge exhibit higher edge density.

The problem of reconstruction from the projected graph has been extensively studied in recent years, both theoretically and empirically. For instance, \cite{WK24} assumes access to a sample of a hypergraph from the underlying distribution and proposes a scoring method to select hyperedges based on their similarity to the sampled hypergraph. Another Bayesian algorithmic approach (i.e., to sample from the posterior distribution given the projected graph) was proposed in \cite{YPP21, LYA23}. In addition, in \cite{CFLL24} a large foundational model is used to recover a weighted hypergraph from a sample of its hyperedges, where each hyperedge is assigned a probability proportional to its weight. In a series of recent progress \cite{BGP24, BGPY25+}, the sharp threshold for reconstructing the hypergraph based on its projected graph in the \emph{noiseless} setting was determined. They further raised the question of whether an analogous reconstruction threshold can be established in the noisy setting. Determining the reconstruction threshold in the noisy setting presents fundamental challenges absent in the noiseless case. In the noiseless model, the projected graph is a deterministic function of the latent hypergraph since every hyperedge manifests as a clique. The problem thus reduces to combinatorially identifying these cliques. In contrast, the noisy setting introduces two layers of uncertainty: hyperedges may fail to generate all their constituent pairwise edges (false negatives), while random background edges can connect vertices not sharing a hyperedge (false positives). This stochasticity destroys the clean clique structure, turning the problem into one of distinguishing subtle statistical deviations in edge density from a random baseline. Consequently, new techniques are required to prove sharp thresholds in the presence of this two-sided noise, making the problem significantly more intricate.

In this work, we not only settle the value of the reconstruction threshold for all $d$, but also establish the value of the (generally higher) threshold for detection, in which the algorithm only wants to detect the existence of group structure. In addition to finding the value of the threshold, we also establish a so-called ``detection-reconstruction gap'' for our problem, a phenomenon where the reconstruction task is strictly harder than detection. This gap, observed in other fundamental statistical problems \cite{DJ04, MST19}, is characterized by a distinct separation between the parameter regimes where detection is possible and where partial reconstruction becomes feasible.

\section{Model definition and main results}

We pose the model as follows: for any constant $d \in \mathbb N$, we denote by $\operatorname{U}=\operatorname{U}_{n,d}$ the collection containing all subsets of $[n]=\{ 1,\ldots,n \}$ with cardinality $d$.  

\begin{defn}{\label{def-noisy-projection}}
    Define $\mathcal H$ to be the random $d$-uniform hypergraph on $[n]$ generated by including each $\Psi\in\operatorname{U}_{n,d}$ in $\mathcal H$ independently with probability $s=n^{-(d-1)+\delta+o(1)}$ for some constant $0\leq \delta<1$\footnote{For any two positive sequences $\{a_n\}$ and $\{b_n\}$, we write equivalently $a_n=O(b_n)$, $b_n=\Omega(a_n)$, $a_n\lesssim b_n$ and $b_n\gtrsim a_n$ if there exists a positive absolute constant $c$ such that $a_n/b_n\leq c$ holds for all $n$. We write $a_n=o(b_n)$, $b_n=\omega(a_n)$, $a_n\ll b_n$, and $b_n\gg a_n$ if $a_n/b_n\to 0$ as $n\to\infty$. We write $a_n =\Theta(b_n)$ if both $a_n=O(b_n)$ and $a_n=\Omega(b_n)$ hold. Without further specification, all asymptotics are taken with respect to $n \to \infty$.}. We denote the law of $\mathcal H$ by $\mu=\mu_{n,d,\delta}$. For simplicity, we will also use $\mathcal{H}$ to denote the set of hyperedges and $|\mathcal{H}|$ the cardinality of $\mathcal{H}$. In addition, for any $d$-uniform hypergraph $\mathcal H$, define $\mathcal P(\mathcal H)$ to be the projection graph on $[n]$ where for all unordered pairs $(i,j)$ with $i\neq j$
    \begin{align}
         (i,j) \in \mathcal P(\mathcal H) \Longleftrightarrow \exists \Psi\in\mathcal H \mbox{ such that } i,j \in\Psi\,.  \label{eq-def-projection}
    \end{align}    
    Consider the random graph $A \triangleq A(\mathcal{H})$ on $[n]$ obtained according to the following rule (denoted by $\Pb=\Pb_{n,d,\delta,\alpha,\beta}$): we first sample $\mathcal H \sim \mu$. Then (conditioned on $\mathcal H$)
    \begin{itemize}
        \item for each $(i,j) \in \mathcal P(\mathcal H)$, we independently include $(i,j)$ in $A$ with probability $p=n^{-1+\alpha+o(1)}\leq 1$ for some $0\leq\alpha\leq 1$;
        \item for each $(i,j) \not\in \mathcal P(\mathcal H)$, we independently include $(i,j)$ in $A$ with probability $q=n^{-1+\beta+o(1)}\leq p$.
    \end{itemize}
\end{defn}
We require the condition $q \le p$, reflecting the natural expectation that edges in $\mathcal{P}(\mathcal{H})$ tend to appear in the graph $A$. Two natural inference problems regarding this model are as follows: (1) the detection problem, i.e., testing $\Pb$ against $\Qb=\Qb_{n,q}$, where $\Qb$ is the law of an \ER graph on $[n]$ with edge density $q=n^{-1+\beta+o(1)}$; (2) the reconstruction problem, i.e., estimating the original hypergraph $\mathcal H$ from the observation $A(\mathcal H)$. 

In a series of recent progress \cite{BGP24, BGPY25+}, the sharp threshold for reconstructing the hypergraph $\mathcal H$ based on $A(\mathcal H)$ in the \emph{noiseless} setting $p=1,q=0$ was determined. In particular, it was shown in \cite{BGPY25+} that when $\delta>\tfrac{d-1}{d+1}$ it is information theoretically impossible to reconstruct a positive fraction of $\mathcal H$ and when $\delta<\tfrac{d-1}{d+1}$ there exists an efficient algorithm that reconstructs a $1-o(1)$ fraction of $\mathcal H$. They further raised the question of whether an analogous reconstruction threshold can be established in the noisy setting. Our main result can be informally summarized as follows:
\begin{thm}{\label{MAIN-THM-informal}}
    \begin{enumerate}
        \item[(1)] (Detection) Suppose $\beta>2(\delta+\alpha)-1$. Then it is information theoretically impossible to distinguish $\Pb$ from $\Qb$. Conversely, if $\beta<2(\delta+\alpha)-1$, then there exists an efficient algorithm to distinguish $\Pb$ from $\Qb$. 
        \item[(2)] (Partial reconstruction) Suppose either $p=o(1)$ or $\delta>\tfrac{d-1}{d+1}$ or $\beta>\tfrac{2\delta}{d(d-1)}+\tfrac{d-2}{d}$. Then it is informationally impossible to reconstruct a positive fraction of $\mathcal H$. Conversely, if $p=\Omega(1),\delta<\tfrac{d-1}{d+1}$ and $\beta<\tfrac{2\delta}{d(d-1)}+\tfrac{d-2}{d}$, then there exists an efficient algorithm that successfully reconstructs a positive fraction of $\mathcal H$. 
        \item[(3)] (Almost-exact reconstruction) Suppose either $p=1-\Omega(1)$, $\delta>\tfrac{d-1}{d+1}$ or $\beta>\tfrac{2\delta}{d(d-1)}+\tfrac{d-2}{d}$. Then it is informationally impossible to reconstruct a $1-o(1)$ fraction of $\mathcal H$. Conversely, if $p=1-o(1),\delta<\tfrac{d-1}{d+1}$ and $\beta<\tfrac{2\delta}{d(d-1)}+\tfrac{d-2}{d}$, then there exists an efficient algorithm that successfully reconstructs a $1-o(1)$ fraction of $\mathcal H$.
    \end{enumerate}
\end{thm}
Notably, our result reveals a \emph{detection-reconstruction gap} phenomenon in this problem where when $p=\Theta(1)$ (which implies that $\alpha=1$), detection is possible for all $\beta,\delta\in(0,1)$ but (partial) reconstruction is possible if and only if $\delta<\tfrac{d-1}{d+1}$ and $\beta<\tfrac{2\delta}{d(d-1)}+\tfrac{d-2}{d}$. See Figure~\ref{fig:summary} for a summary of the regions where detection/reconstruction is possible. The reader may wonder whether the detection-recovery gap could be eliminated by choosing an alternative null distribution whose edge density matched that of the planted model $\Pb$. We show in Appendix~\ref{sec:alternative-detection-prob} that this adjustment closes the gap partially but not entirely. In particular, detection remains possible \emph{at least} whenever $\delta,\beta<\tfrac{d-1}{d+1}$, while for $\delta<\tfrac{d-1}{d+1}$ we always have $\tfrac{d-1}{d+1}>\tfrac{d-2}{d}+\tfrac{2\delta}{d(d-1)}$. 
\begin{figure}[ht!]
    \centering
    \includegraphics[height=9cm,width=11.5cm]{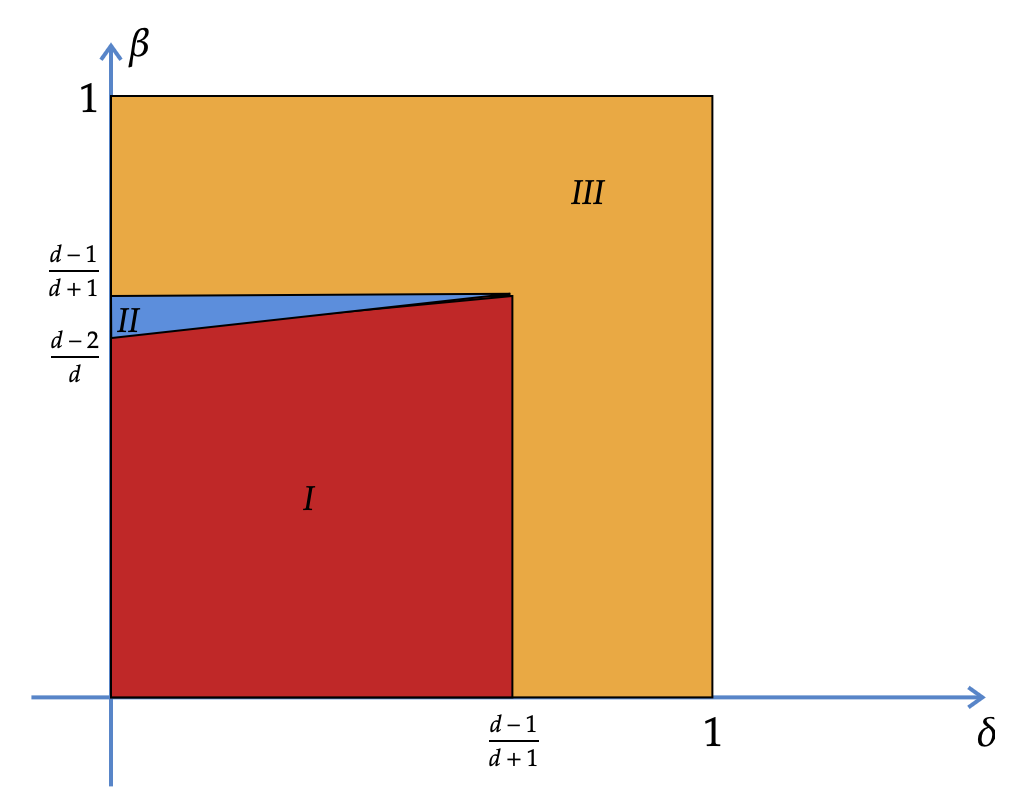}
    \caption{Summary of the detection-reconstruction gap when $p=\Theta(1)$. Red regime $I$: reconstruction of $\mathcal H$ is possible; Blue regime $II$: detection with a better null distribution $\widetilde{\Qb}$ (that matches the average edge-density) is possible, but reconstruction of $\mathcal H$ is impossible; Yellow regime $III$: detection with the simple null distribution $\Qb$ is possible, but detection with a better null distribution $\widetilde{\Qb}$ is unknown.}
    {\label{fig:summary}}
\end{figure}

\subsection{The detection threshold}

In this subsection, we state precisely the results that establish Item~(1) of Theorem~\ref{MAIN-THM-informal}.
\begin{defn}{\label{def-strong-weak-detection}}
    For two sequences of probability measures $\Pb=\Pb_n$ and $\Qb=\Qb_n$ on the same sample $X=X_n$. A test statistic $f=f(X)\in\{ 0,1 \}$ is said to achieve
    \begin{itemize}
        \item Strong detection, if $\Pb(f(X)=0)+\Qb(f(X)=1)=o(1)$ as $n\to\infty$.
        \item Weak detection, if $\Pb(f(X)=0)+\Qb(f(X)=1)=1-\Omega(1)$ as $n\to\infty$.
    \end{itemize}
\end{defn}
We first show that when $\beta<2(\delta+\alpha)-1$, there exists a test statistic that is efficient and achieves strong detection. Consider the following edge-counting statistic 
\begin{equation}{\label{eq-def-edge-counting}}
    f=f(A) = \sum_{1\leq i<j\leq n} \mathbf{1}_{\{(i,j)\in A\}} \,.
\end{equation}
It is clear that $f$ can be computed in time $O(n^{2})$.
\begin{proposition}{\label{prop-detect-alg}}
    Suppose that $\beta<2(\delta+\alpha)-1$. Then we have 
    \begin{equation}{\label{eq-success-edge-count}}
    \begin{aligned}
        &\Pb\Big( f(A) > q \tbinom{n}{2} + \tfrac{1}{2} ps\cdot \tbinom{d}{2} \tbinom{n}{d}\Big) = 1-o(1) \,; \\
        &\Qb\Big( f(A) < q \tbinom{n}{2} + \tfrac{1}{2} ps\cdot \tbinom{d}{2} \tbinom{n}{d} \Big) = 1-o(1) \,.
    \end{aligned}
    \end{equation}
    In particular, the thresholding statistic on $f(A)$ achieves strong detection between $\Pb$ and $\Qb$.
\end{proposition}
The proof of Proposition~\ref{prop-detect-alg} is incorporated in Section~\ref{sec:proof-prop-1.4}. We now turn to showing that whenever $\beta>2(\delta+\alpha)-1$, \emph{no} test, strong or weak, can reliably distinguish $\mathbb P$ from $\mathbb Q$. It is well-known that the testing error is governed by the total variation distance between the null and alternative distributions. Our next result establishes the desired bound on this total variation distance in the regime $\beta>2(\delta+\alpha)-1$.
\begin{proposition}{\label{prop-bound-TV}}
    When $\beta>2(\delta+\alpha)-1$, we have $\operatorname{TV}(\Pb,\Qb)=o(1)$. Thus, it is informationally impossible to achieve weak detection between $\Pb$ and $\Qb$.
\end{proposition}
The proof of Proposition~\ref{prop-bound-TV} is incorporated in Section~\ref{sec:proof-prop-1.5}. Our strategy is as follows: first we combine the standard Pinsker's inequality and replica's trick to bound $\operatorname{TV}(\Pb,\Qb)$ by $|E(\mathcal P(\mathcal H)) \cap E(\mathcal P(\mathcal H'))|$ where $\mathcal H,\mathcal H'$ are independent copies from $\mu$. Then by an application of Talagrand's concentration inequality we can argue that $|E(\mathcal P(\mathcal H)) \cap E(\mathcal P(\mathcal H'))|$ tightly concentrates around its mean, which yields the desired total variation bound.

\subsection{The reconstruction threshold}

In this subsection, we state precisely the results that confirm Items~(2) and (3) of Theorem~\ref{MAIN-THM-informal}.

\begin{defn}{\label{def-strong-weak-reconstruction}}
    Recall the definition of $s$ in Definition \ref{def-noisy-projection}. A test statistic $\widehat{\mathcal H}=\widehat{\mathcal H}(A)$ is said to achieve
    \begin{itemize}
        \item Partial reconstruction, if $\mathbb E_{\Pb}\big[ |\mathcal H \triangle \widehat{\mathcal H}| \big]=(1-\Omega(1))\cdot s\tbinom{n}{d}$ as $n\to\infty$.
        \item Almost-exact reconstruction, if $\mathbb E_{\Pb}\big[ |\mathcal H \triangle \widehat{\mathcal H}| \big]=o(1)\cdot s\tbinom{n}{d}$ as $n\to\infty$.
    \end{itemize}
\end{defn}
We first show the positive results on partial and almost-exact reconstruction. Consider the ``clique'' estimator 
\begin{equation}{\label{eq-def-widehat-H}}
    \widehat{\mathcal H}:= \big\{ \Psi\in\operatorname{U} : A_{i,j}=1 \mbox{ for all } i,j \in \Psi \big\}\,.
\end{equation}
\begin{proposition}{\label{prop-succeess-widehat-H}}
    Suppose $\delta<\tfrac{d-1}{d+1},\beta<\tfrac{2\delta}{d(d-1)}+\tfrac{d-2}{d}$ and $p=1-o(1)$. We have
    \begin{equation*}
        \mathbb E_{\Pb}\Big[ \big| \widehat{\mathcal H} \triangle \mathcal H \big| \Big] = o(1) \cdot s\tbinom{n}{d}  \,.
    \end{equation*}
    In addition, when $\delta<\tfrac{d-1}{d+1},\beta<\tfrac{2\delta}{d(d-1)}+\tfrac{d-2}{d}$ and $p=\Omega(1)$, we have
    \begin{equation*}
        \mathbb E_{\Pb}\Big[ \big| \widehat{\mathcal H} \triangle \mathcal H \big| \Big] = (1-\Omega(1)) \cdot s\tbinom{n}{d}  \,.
    \end{equation*}
\end{proposition}
The proof of Proposition~\ref{prop-succeess-widehat-H} is incorporated in Section~\ref{sec:proof-prop-1.7}. Next we show the impossibility of weak reconstruction. Note that when $\delta>\tfrac{d-1}{d+1}$, by \cite[Theorem~1.1]{BGPY25+}, it is informationally impossible to achieve weak reconstruction even when $p=1,q=0$. Thus, without loss of generality, we assume that $\delta\leq\tfrac{d-1}{d+1}$.
\begin{proposition}{\label{prop-information-recovery-lower-bound}}
    \begin{itemize}
        \item[(1)] Suppose that $p=o(1)$. Then for any estimator $\widehat{\mathcal H}=\widehat{\mathcal H}(A)$ we have 
        \begin{align*}
            \mathbb E_{\mathcal H,A \sim \Pb}\Big[ \big|\widehat{\mathcal H}(A) \triangle \mathcal H\big| \Big]=(1-o(1)) \cdot s\tbinom{n}{d}
        \end{align*}
        even when $q=0$.
        \item[(2)] Suppose $q=n^{-1+\beta+o(1)}$ where $\beta>\tfrac{2\delta}{d(d-1)}+\tfrac{d-2}{d}$. Then for any estimator $\widehat{\mathcal H}=\widehat{\mathcal H}(A)$ we have 
        \begin{align*}
            \mathbb E_{\mathcal H,A \sim \Pb}\Big[ \big|\widehat{\mathcal H}(A) \triangle \mathcal H\big| \Big]=(1-o(1)) \cdot s\tbinom{n}{d}
        \end{align*}
        even when $p=1$.
    \end{itemize}
\end{proposition}
The proof of Proposition~\ref{prop-information-recovery-lower-bound} is given in Section~\ref{sec:proof-prop-1.8}. Item~(1) follows from a suitable generalization of Fano’s inequality, while Item~(2) is established via a Bayesian argument. 

We then show the impossibility of strong reconstruction. In light of Proposition~\ref{prop-information-recovery-lower-bound}, it remains only to prove the following result, whose proof appears in Section~\ref{sec:proof-prop-1.9}.
\begin{proposition}{\label{prop-information-almost-exact-recovery-lower-bound}}
    Suppose that $p=1-\Omega(1)$. Then for any estimator $\widehat{\mathcal H}=\widehat{\mathcal H}(A)$ we have 
    \begin{align*}
        \mathbb E_{\mathcal H,A \sim \Pb}\big[ |\widehat{\mathcal H}(A) \triangle \mathcal H| \big]=\Omega(1) \cdot s\tbinom{n}{d}
    \end{align*}
    even when $q=0$.
\end{proposition}

\section{Proof of Proposition~\ref{prop-detect-alg}}{\label{sec:proof-prop-1.4}}

Recall \eqref{eq-def-edge-counting}. Under $\Qb$, $f(A)$ is the sum of $\tbinom{n}{2}$ i.i.d.\ random variables with law $\operatorname{Ber}(q)$. It is clear that $\Eb_{\Qb}[f(A)]=q\tbinom{n}{2}$ and $\Var_{\Qb}[f(A)]\leq q\tbinom{n}{2}$. By Chebyshev's inequality, we get (also recall that $q=n^{-1+\beta+o(1)}$, $p=n^{-1+\alpha+o(1)}$ and $s=n^{-(d-1)+\delta+o(1)}$)
\begin{equation*}
    \Qb\Big( f(A) \geq q\tbinom{n}{2} + \tfrac{1}{2} ps\cdot \tbinom{d}{2} \tbinom{n}{d} \Big)\leq \frac{\Var_{\Qb}[f(A)]}{\big[\tfrac{1}{2} ps\cdot \tbinom{d}{2} \tbinom{n}{d}\big]^2} = O\big(n^{1+\beta-2(\alpha+\delta+o(1))}\big)=o(1) \,,
\end{equation*}
where the last equality follows from $\beta<2(\alpha+\delta)-1$. Now we focus on the behavior of $f(A)$ under $\Pb$. Note that conditioning on $\mathcal H$, we have (the binomial variables are independent)
\begin{equation}{\label{eq-edge number}}
        |E(A(\mathcal H))| \overset{d}{=} \mathbf B\big( |E(\mathcal P(\mathcal H))|,p \big) + \mathbf B\big( \tbinom{n}{2}-|E(\mathcal P(\mathcal H))|,q \big) \,.
\end{equation}
Thus, we first show 
\begin{align}\label{eq-est-edges-projection}
    \Pb\Big(0.9s\tbinom{d}{2}\tbinom{n}{d}\leq |E(\mathcal P(\mathcal H))|\leq 1.1s\tbinom{d}{2}\tbinom{n}{d} \Big) = 1-o(1) \,.
\end{align}
Let $M_k(\mathcal H)$ be the set of $(i,j)\in\mathcal P(\mathcal H)$ such that there exist at least $k$ hyperedges in $\mathcal H$ that contain $(i,j)$. In addition, define 
\begin{equation}{\label{eq-def-mathcal-E}}
    \mathcal E=\mathcal E_1 \cap \mathcal E_2 \,,
\end{equation}
where
\begin{equation}{\label{eq-def-mathcal-E-1}}
    \mathcal E_1=\Big\{ |\mathcal{H}|=(1+o(1))s\tbinom{n}{d} \Big\}
\end{equation}
and
\begin{equation}{\label{eq-def-mathcal-E-2}}
    \mathcal E_2=\Big\{ \sum_{k \geq 2} k|M_k(\mathcal{H})|=o(1) \cdot s\tbinom{n}{d} \Big\} \,.
\end{equation}
Clearly we have $\Pb_{\mathcal H\sim\mu}(\mathcal E_1)=1-o(1)$. In addition, note that $(i,j)\in M_k(\mathcal H)$ implies that there exist $k$ different $\Psi_1,\ldots,\Psi_k \in\operatorname{U}$ with $\Psi_l \in\mathcal H$ and $(i,j)\in E(\mathcal P(\Psi_l))$ for $1 \leq l\leq k$. Thus, by a simple union bound we get that (note that the enumeration of $k$ different $\Psi_1,\ldots,\Psi_k \in\operatorname{U}$ with $(i,j)\in E(\mathcal P(\Psi_l))$ is bounded by $n^{(d-2)k}$)
\begin{align*}
    \Pb_{\mathcal H\sim\mu}\big( (i,j)\in M_k(\mathcal H) \big) \leq n^{(d-2)k} \cdot s^{k} \leq n^{(-1+\delta+o(1))k} \,.
\end{align*}
This implies that 
\begin{align*}
    \mathbb E_{\mathcal H \sim\mu}\Big[ \sum_{k \geq 2} k|M_k(\mathcal{H})| \Big] \leq \sum_{k \geq 2} kn^2 \cdot n^{(-1+\delta+o(1))k}  = n^{2\delta+o(1)} = o(1) \cdot s\tbinom{n}{d} \,.
\end{align*}
Thus, from a simple Markov inequality we get that $\Pb_{\mathcal H\sim\mu}(\mathcal E_2)=1-o(1)$. Thus, we have
\begin{equation}{\label{eq-mathcal-E-is-typical}}
    \Pb_{\mathcal H\sim\mu}(\mathcal E)=1-o(1) \,.
\end{equation}
Also, on $\mathcal E$ we get from the deterministic bound that
\begin{align*}
    |E(\projH)|\leq \tbinom{d}{2}|\mathcal{H}|\leq 1.1s\tbinom{d}{2}\tbinom{n}{d}
\end{align*}
and
\begin{align*}
    |E(\mathcal P(\mathcal H))| \geq \tbinom{d}{2} |\mathcal H| - \sum_{k \geq 2} k|M_k(\mathcal{H})| \geq 0.9s \tbinom{d}{2}\tbinom{n}{d} \,,
\end{align*}
which yields \eqref{eq-est-edges-projection}. Now conditioning on $\projH$ satisfying the event in \eqref{eq-est-edges-projection}, by \eqref{eq-edge number} and a standard binomial tail estimate, we obtain that with probability $1-o(1)$ under $\Pb$ 
\begin{equation*}
    |E(A(\mathcal H))| \geq q\tbinom{n}{2}+0.9(p-q) s\tbinom{d}{2}\tbinom{n}{d} - \sqrt{(ps\tbinom{d}{2}\tbinom{n}{d}+q\tbinom{n}{2})\log n} \geq q\tbinom{n}{2}+\tfrac{1}{2} ps\tbinom{d}{2}\tbinom{n}{d} \,, 
\end{equation*}
where in the last inequality we used $n^{\frac{1+\beta}{2}}<n^{\delta+\alpha-\Omega(1)}$. As a result, \eqref{eq-success-edge-count} holds.

\section{Proof of Proposition~\ref{prop-bound-TV}}{\label{sec:proof-prop-1.5}}

Let $L(A)=\tfrac{\Pb(A)}{\Qb(A)}$ denote the likelihood ratio. By Pinsker's inequality we have
\begin{align*}
    \operatorname{TV}(\Pb,\Qb) \leq \mathbb E_{\Qb}\big[ L(A)^2 \big]-1 \,.
\end{align*}
Thus, to prove Proposition~\ref{prop-bound-TV}, it suffices to show the following result.
\begin{lemma}{\label{prop-info-detection-lower-bound}}
    When $\beta>2(\delta+\alpha)-1$, we have 
    \begin{align}{\label{eq-goal-detection-lower-bound}}
        \mathbb E_{\Qb}\big[ L(A)^2 \big] = 1+o(1) \,.
    \end{align}
\end{lemma}
To prove Lemma~\ref{prop-info-detection-lower-bound}, we first bound $\mathbb E_{\Qb}\big[ L(A)^2 \big]$ by the following lemma.
\begin{lemma}{\label{lem-second-moment-relax}}
    We have
    \begin{equation}{\label{eq-second-moment-relax-1}}
        \mathbb E_{\Qb}\Big[L(A)^2\Big]\leq\mathbb E_{\mathcal H,\mathcal H' \sim\mu\otimes\mu}\Bigg[ \exp\Big(n^{-1-\beta+2\alpha+o(1)}\cdot \big| E(\mathcal P(\mathcal H)) \cap E(\mathcal P(\mathcal H')) \big| \Big) \Bigg] \,.
    \end{equation}
\end{lemma}
\begin{proof}
    Note that  
    \begin{align*}
        L(A) &= \frac{\Pb(A)}{\Qb(A)} = \mathbb E_{\mathcal H \sim\mu}\left[ \frac{\Pb(A\mid\mathcal H)}{\Qb(A)} \right] \\
        &= \mathbb E_{\mathcal H \sim\mu}\left[ \prod_{1 \leq i<j\leq n} \frac{\Pb(A_{i,j}\mid\mathcal H)}{\Qb(A_{i,j})} \right] = \mathbb E_{\mathcal H \sim\mu}\left[ \prod_{(i,j) \in E(\mathcal P(\mathcal H))} \ell(A_{i,j}) \right] \,,
    \end{align*}
    where
    \begin{align*}
        \ell(A_{i,j}) = \mathbf{1}_{\{A_{i,j}=1\}} \cdot \frac{n^{-1+\alpha+o(1)}}{n^{-1+\beta+o(1)}} + \mathbf{1}_{\{A_{i,j}=0\}} \cdot \frac{1-n^{-1+\alpha+o(1)}}{1-n^{-1+\beta+o(1)}} \,.
    \end{align*}
    We then have (note that $\mathbb E_{\Qb}[\ell(A_{i,j})]=1$)
    \begin{align}
        \mathbb E_{\Qb}\Big[L(A)^2\Big] &=\mathbb E_{\mathcal H,\mathcal H' \sim\mu\otimes\mu}\left\{ \mathbb E_{A\sim\Qb}\left[ \frac{\Pb(A\mid\mathcal H)}{\Qb(A)} \cdot \frac{\Pb(A\mid\mathcal H')}{\Qb(A)} \right] \right\} \nonumber \\
        &=\mathbb E_{\mathcal H,\mathcal H' \sim\mu\otimes\mu}\Bigg\{ \mathbb E_{A\sim\Qb}\Big[ \prod_{(i,j) \in E(\mathcal P(\mathcal H))} \ell(A_{i,j}) \cdot \prod_{(i,j) \in E(\mathcal P(\mathcal H'))} \ell(A_{i,j}) \Big] \Bigg\} \nonumber \\
        &= \mathbb E_{\mathcal H,\mathcal H' \sim\mu\otimes\mu}\Bigg\{ \prod_{(i,j) \in E(\mathcal P(\mathcal H)) \cap E(\mathcal P(\mathcal H'))} \mathbb E_{A\sim\Qb}\Big[ \ell(A_{i,j})^2 \Big] \Bigg\} \,.  \label{eq-second-moment-calc-1}
    \end{align}
    Finally, from straightforward calculation we have
    \begin{align*}
        \mathbb E_{A\sim\Qb}\Big[ \ell(A_{i,j})^2 \Big] &=\left( \frac{n^{-1+\alpha+o(1)}}{n^{-1+\beta+o(1)}} \right)^2 \cdot n^{-1+\beta+o(1)} + \left( \frac{1-n^{-1+\alpha+o(1)}}{1-n^{-1+\beta+o(1)}} \right)^2 \cdot (1-n^{-1+\beta+o(1)}) \\
        &\leq n^{-1-\beta+2\alpha+o(1)}+1
        \leq \,\exp\big(n^{-1-\beta+2\alpha+o(1)}\big) \,,
    \end{align*}
    where the second line follows from the assumption $\beta\leq\alpha\leq 1$. Plugging the above estimation into \eqref{eq-second-moment-calc-1}, we get \eqref{eq-second-moment-relax-1}.
\end{proof}

Based on Lemma~\ref{lem-second-moment-relax}, it suffices to bound the tail probability of $|E(\mathcal P(\mathcal H)) \cap E(\mathcal P(\mathcal H'))|$, as shown in the following lemma. 
\begin{lemma}{\label{lem-est-tail-intersection}}
    We have
    \begin{equation}{\label{eq-tail-est-intersection}}
        \mathbb P_{\mathcal H,\mathcal H' \sim\mu\otimes\mu}\Big( \big| E(\mathcal P(\mathcal H)) \cap E(\mathcal P(\mathcal H')) \big| \geq n^{2\delta+o(1)} + t n^{\delta+o(1)} \Big) \leq \exp\big(-\Omega(1)\cdot t^2\big) \,.
    \end{equation}
\end{lemma}

To prove Lemma~\ref{lem-est-tail-intersection}, we begin by recalling the relevant definitions and notation for Talagrand's concentration inequality.
\begin{defn}
    Let $\Omega=\prod_{i=1}^{N}\Omega_i$ be a product space. For a function $h:\Omega\to\mathbb R$, we call $h$ to be Lipschitz, if $|h(x)-h(y)|\leq 1$ whenever $x,y$ differ in at most one coordinate. Furthermore, for an increasing function $f:\mathbb R_{\geq 0}\to\mathbb R_{\geq 0}$, we say that $h$ is $f$-certifiable if, whenever $h(x)\geq s$, there exists $I\subset\{1,\ldots,N\}$ with $|I|\leq f(s)$ such that all $y\in\Omega$ that agree with $x$ on the coordinates in $I$ satisfy $h(y)\geq s$.
\end{defn}
A special case of Talagrand’s concentration inequality can be stated as follows.
\begin{thm}{\label{thm-Talagrand}}
    Let $\Pb$ be any product measure on a sample space $\Omega$, and let $h$ be a Lipschitz function that is $f$-certifiable. Writing $X=h(\omega)\,(\omega\in\Omega)$, for any $b,t>0$ we have:
    \begin{align*}
        \Pb\big(X\leq b-t\sqrt{f(b)}\big)\cdot\Pb\big(X\geq b\big)\leq e^{-t^2/4}  \,.
    \end{align*}
\end{thm}
The proof of this inequality can be found in, e.g., \cite[Section 7.7]{AS16}. We now proceed to the proof of Lemma~\ref{lem-est-tail-intersection}. 
\begin{proof}[Proof of Lemma~\ref{lem-est-tail-intersection}]
    For our purpose, we consider $\Omega=\{0,1\}^{2\binom{n}{d}}$ and $\Pb$ be the product measure of $2\binom{n}{d}$ Bernoulli variables with parameter $s=n^{-(d-1)+\delta+o(1)}$. Let
    \begin{align*}
        X = h(\mathcal H,\mathcal H')=\big|E(\mathcal P(\mathcal H)) \cap E(\mathcal P(\mathcal H'))\big|/\tbinom{d}{2}, 
    \end{align*}
    where we view $(\mathcal H,\mathcal H')$ as a vector in $\{0,1\}^{2\binom{n}{d}}$. It is straightforward to verify that $h$ is Lipschitz. Furthermore, we claim that the function $h$ defined above is $f$-certifiable, where $f(s)=2\tbinom{d}{2}s \leq d^2 s$. 
    In fact, if $h(\mathcal H,\mathcal H')\geq s$ for some $(\mathcal H,\mathcal H')\in\Omega$, then there exists $E=\{(i,j):1\leq i<j\leq n\}$ such that for all $e\in E$ there exist $\zeta_e,\zeta_e' \in \tbinom{[n]}{d}$ with $\mathcal H(\zeta_e)=\mathcal H'(\zeta'_e)=1$. Denote $I=\{ \zeta_e:e\in E \} \cap \{ \zeta'_e: e \in E \}$. Then for all $(\widetilde{\mathcal H},\widetilde{\mathcal H}')\in \{0,1\}^{2\binom{n}{d}}$ that agrees with $(\mathcal H,\mathcal H')$ on $I$, we have $e \in E(\mathcal P(\widetilde{\mathcal H})) \cap E(\mathcal P(\widetilde{\mathcal H}'))$ and thus $h(\widetilde{\mathcal H},\widetilde{\mathcal H}') \geq s$. As a result, Theorem~\ref{thm-Talagrand} yields that for any $b,t\geq 0$,
    \begin{align*}
        \Pb\big(X\leq b-td\sqrt{b}\big)\cdot\Pb(X\geq b)\leq e^{-t^2/4} \,.
    \end{align*}
    Taking $b=\operatorname{Median}(X)$ to be the median of $X$, we then get that
    \begin{equation}{\label{eq-tail-intersection-variable-median-1}}
    \begin{aligned}
        \Pb\Big( X \leq \operatorname{Median}(X)-t \Big) \leq 2e^{-t^2/4d^2\operatorname{Median}(X)} \,.\notag
    \end{aligned}
    \end{equation}
    In addition, taking $b=\operatorname{Median}(X)+r$ and $t=r/d\sqrt{b}$ we get that
    \begin{equation}{\label{eq-tail-intersection-variable-median-2}}
    \begin{aligned}
        \Pb\Big( X \geq \operatorname{Median}(X)+r \Big) \leq 2e^{-r^2/4d^2(\operatorname{Median}(X)+r)} \leq 2e^{-r/8d^2} + 2e^{-r^2/8d^2\operatorname{Median}(X)} \,.\notag
    \end{aligned}
    \end{equation}
    Thus we have
    \begin{align}
        \Big|\mathbb E[X]-\operatorname{Median}(X)\Big| &\leq\mathbb E\Big[\big|X-\operatorname{Median}(X)\big|\Big] \nonumber \\
        &=\int_{0}^{\infty} \Pb\Big(\big|X-\operatorname{Median}(X)\big|\geq t\Big) \mathrm{d}t \nonumber \\
        &\leq \int_{0}^{\infty} 4\exp\Big( -\tfrac{t^2}{8d^2\operatorname{Median}(X)} \Big) +2\exp\big(-\tfrac{t}{8d^2}\big) \mathrm{d}t \nonumber \\
        &= 4\sqrt{2\pi}d\sqrt{\operatorname{Median}(X)}+16d^2. \label{eq-est-median-minus-expectation}
    \end{align}
    By the independence of $\mathcal H$ and $\mathcal H'$, we can easily get
    \begin{equation*}
        \E[X]=\tbinom{n}{2}\big(\tbinom{n-2}{d-2}s\big)^2/\tbinom{d}{2}=n^{2\delta+o(1)} \,.
    \end{equation*}
    Thus, we have (also note that $\mathbb E[X]=n^{2\delta+o(1)}$ and \eqref{eq-est-median-minus-expectation} imply that $\operatorname{Median}(X)=\mathbb E[X]+\Omega(1)\cdot\sqrt{\mathbb E[X]}$)
    \begin{equation}{\label{eq-tail-intersection-variable-mean}}
        \Pb\Big(X\geq\mathbb E[X]+t\sqrt{\mathbb E[X]} \Big) \leq 2e^{-\Omega(1)\cdot t^2} \,.
    \end{equation}
    Plugging $\mathbb E[X] = n^{2\delta+o(1)}$ into \eqref{eq-tail-intersection-variable-mean} yields the desired result. 
\end{proof}

Now we can complete the proof of Lemma~\ref{prop-info-detection-lower-bound}. 
\begin{proof}[Proof of Lemma~\ref{prop-info-detection-lower-bound}] 
    Let $Y(\mathcal{H}, \mathcal{H}')=\big|E(\mathcal P(\mathcal H)) \cap E(\mathcal P(\mathcal H'))\big|$ and recall that $1+\beta>2(\delta+\alpha)$. 
    Applying Lemma \ref{lem-est-tail-intersection}, we obtain the following estimate
    \begin{align*}
        &\quad\ \E_{\mathcal{H},\mathcal{H}'\sim \mu\otimes\mu}\Big[\exp\big(n^{-1-\beta+2\alpha+o(1)}\cdot Y\big)\cdot \11\big{\{Y\geq 2 n^{2\delta+o(1)}\big\}}\Big]\\
        &=\int_{2n^{2\delta+o(1)}}^{\infty}n^{-1-\beta+2\alpha+o(1)}\exp\big(n^{-1-\beta+2\alpha+o(1)}\cdot y\big)\cdot 
        \P_{\mathcal{H},\mathcal{H}'\sim \mu\otimes\mu}(Y\geq y)\mathrm{d}y\\
        \overset{\text{Lemma}~\ref{lem-est-tail-intersection}}&{\leq}
        \int_{n^{\delta+o(1)}}^{\infty}\exp\Big(n^{-1-\beta+2\alpha+o(1)}\big(n^{2\delta+o(1)}+tn^{\delta+o(1)}\big)\Big)\exp(-ct^2)\mathrm{d}t\\
        &=o\Big(\exp\big(n^{-1-\beta+2(\alpha+\delta)+o(1)}\big)\Big)=o(1). 
    \end{align*}
    Then, we have (recall that $1+\beta=2(\alpha+\delta)+\Omega(1)$)
    \begin{equation*}
        \E_{\mathcal{H},\mathcal{H}'\sim \mu\otimes\mu}\Big[\exp\big(n^{-1-\beta+2\alpha+o(1)}\cdot Y\big)\Big] \leq
        \exp\big(2\cdot n^{-1-\beta+2(\alpha+\delta)+o(1)}\big)+o(1)=1+o(1). 
    \end{equation*}
    Combining this bound with Lemma~\ref{lem-second-moment-relax} gives \eqref{eq-goal-detection-lower-bound}. 
\end{proof}

\section{Proof of Proposition~\ref{prop-succeess-widehat-H}}{\label{sec:proof-prop-1.7}}

Firstly, note that for all $\Psi\subset [n]$ with $|\Psi|=d$ we have
\begin{align*}
    &\Pb(\Psi\in\widehat{\mathcal H}\cap\mathcal H)=p^{\binom{d}{2}} \cdot s=\Theta(1)\cdot s=\Theta(1) \cdot \Pb(\Psi\in\mathcal H) \mbox{ when } p=\Omega(1) \,, \\
    &\Pb(\Psi\in\widehat{\mathcal H}\cap\mathcal H)=p^{\binom{d}{2}} \cdot s=(1-o(1))\cdot s=(1-o(1))\cdot\Pb(\Psi\in\mathcal H) \mbox{ when } p=1-o(1) \,.
\end{align*}
This implies that 
\begin{align*}
    &\mathbb E\left[|\mathcal{H}\setminus\widehat{\mathcal{H}}|\right] = (1-\Omega(1))\cdot s\tbinom{n}{d} \text{ when }p=\Omega(1)\,, \\
    &\mathbb E\left[|\mathcal{H}\setminus\widehat{\mathcal{H}}|\right] = o(1)\cdot s\tbinom{n}{d} \text{ when }p=1-o(1)\,.
\end{align*}
Thus, it suffices to show that when $\delta<\tfrac{d-1}{d+1}$ and $\beta<\tfrac{2\delta}{d(d-1)}+\tfrac{d-2}{d}$, for all $0\leq p\leq 1$ we have
\begin{align*}
    \Pb(\Psi\in\widehat{\mathcal H}\mid \Psi\not\in\mathcal H)=o(1)\cdot s \,.
\end{align*}
We follow the approach in \cite{BGP24}. Fix any $\Psi\subset[n]$ with $|\Psi|=d$, denote 
\begin{align*}
    \mathcal U=\mathcal U(\mathcal H)=\Big\{ A \subset \Psi: |A| \geq 2, \exists \Phi \in \mathcal H, A = \Phi \cap \Psi \Big\} \,.
\end{align*}
Then we have
\begin{align*}
    \Pb\big(\Psi\in\widehat{\mathcal H}\mid \Psi\not\in\mathcal H\big)= \sum_{\mathcal U:\Psi\not\in\mathcal U} \Pb\big(\Psi\in\widehat{\mathcal H};\mathcal U\big)= \sum_{\mathcal U:\Psi\not\in\mathcal U} \Pb(\mathcal U) \Pb\big(\Psi\in\widehat{\mathcal H}\mid\mathcal U\big) \,.
\end{align*}
Denote $\mathcal U=\{ A_1,\ldots,A_k \}$. We then get from a simple union bound that
\begin{align*}
    \Pb(\mathcal U)=\prod_{1 \leq \ell \leq k} sn^{d-|A_\ell|}=\prod_{1 \leq \ell \leq k} n^{1+\delta-|A_\ell|+o(1)} \,.
\end{align*}
Denoting $(x)_+\triangleq x\mathbf{1}_{x\geq 0}$ the positive part of $x$, we have
\begin{align*}
    \Pb\big( \Psi\in\widehat{\mathcal H}\mid\mathcal U \big)&=\Pb\big( (i,j)=1, \forall (i,j)\in E(\Psi)\setminus E(A_\ell) \big)  \\
    &\leq q^{ (\binom{d}{2} - \sum_{1 \leq \ell \leq k} \binom{|A_\ell|}{2})_+ }=n^{ (-1+\beta)(\binom{d}{2} - \sum_{1 \leq \ell \leq k} \binom{|A_\ell|}{2})_+ + o(1) } \,.
\end{align*}
Thus, we get that
\begin{align*}
    \Pb\big(\Psi\in\widehat{\mathcal H}\mid \Psi\not\in\mathcal H\big) \leq \sum_{\substack{(A_1,\ldots,A_k)\\A_\ell\subset\Psi,\Psi\neq A_\ell}} n^{-\omega(A_1,\ldots,A_k)+o(1)} \,,
\end{align*}
where 
\begin{align*}
    \omega(A_1,\ldots,A_k)&=\sum_{1 \leq \ell \leq k}\big( |A_\ell|-1-\delta \big)+(1-\beta)\Big( \tbinom{d}{2} - \sum_{1 \leq i \leq k} \tbinom{|A_\ell|}{2} \Big)_+ \\
    &> d-1-\delta+\sum_{1 \leq \ell \leq k}\big( |A_\ell|-1-\delta \big)- \big(\tfrac{2}{d}-\tfrac{2\delta}{d(d-1)}\big) \sum_{1 \leq \ell \leq k} \tbinom{|A_i|}{2} \,,
\end{align*}
where the inequality follows from $\beta<\tfrac{d-2}{d}+\tfrac{2\delta}{d(d-1)}$. Note that $2 \leq |A_\ell| \leq d-1$ and $\delta<\tfrac{d-1}{d+1}$, we get that 
\begin{align*}
    &\sum_{1 \leq \ell \leq k}\big( |A_\ell|-1-\delta \big)- \big(\tfrac{2}{d}-\tfrac{2\delta}{d(d-1)}\big) \sum_{1 \leq \ell \leq k} \tbinom{|A_\ell|}{2} \\
    =\ &-\tfrac{1}{d} \sum_{1 \leq \ell \leq k} \big(|A_\ell|-d\big) \big((1-\tfrac{\delta}{d-1})|A_\ell|-(1+\delta)\big) \geq 0 \,,
\end{align*}
In conclusion, we obtain that $\omega(A_1,\ldots,A_k)>d-1-\delta$ for all $(A_1,\ldots,A_k)$. Since the possible choice of $(A_1,\ldots,A_k)$ is bounded by a constant depending only on $d$, we get that
\begin{align*}
    \Pb\big(\Psi\in\widehat{\mathcal H}\mid \Psi\not\in\mathcal H\big) \leq O(1) \cdot n^{-(d-1)+\delta-\Omega(1)} = o(1) \cdot s \,,
\end{align*}
as desired.

\section{Proof of Proposition~\ref{prop-information-recovery-lower-bound}}\label{sec:proof-prop-1.8}

We first prove Item~(1). Note that if there is an estimator $\widehat{\mathcal H}$ such that $\mathbb E_{\Pb}[ |\mathcal H \triangle \widehat{\mathcal H}| ]=(1-\Omega(1))s\tbinom{n}{d}$. Then clearly we have
\begin{align*}
    \Pb\Big( |\widehat{\mathcal H} \cap \mathcal H| \geq \Omega(1) s\tbinom{n}{d}, |\widehat{\mathcal H}| \leq (2-\Omega(1))s\tbinom{n}{d} \Big)\geq \Omega(1) \,.
\end{align*}
In addition, note that when $|\widehat{\mathcal H}|\geq s\tbinom{n}{d}$, we can always consider a random subset of $\widehat{\mathcal H}$ with cardinality $s\tbinom{n}{d}$; and when $|\widehat{\mathcal H}|<s \tbinom{n}{d}$ we can always add arbitrary elements into $\widehat{\mathcal H}$; this means that we can generate an estimator $\widehat{\mathcal H}_{\bullet}$ such that $|\widehat{\mathcal H}_{\bullet}|=(1+o(1))s\tbinom{n}{d}$ a.s. and $\mathbb E_{\Pb}[ |\widehat{\mathcal H}_{\bullet} \cap \mathcal H| ]=\Omega(1) s\tbinom{n}{d}$. We will show that this is impossible. Define $\mu'=\mu\big(\cdot\mid |\mathcal H|\leq(1+o(1)) s\tbinom{n}{d}\big)$. Conditioned on $\mathcal{H}\sim\mu'$, we define $\Pb'$ by letting $\Pb'(\cdot\mid\mathcal H)\triangleq\Pb(\cdot\mid\mathcal H)$. It is straightforward to check that $\operatorname{TV}(\Pb',\Pb)=o(1)$. Thus, by data processing inequality, it suffices to show that when $p=o(1),q=0$, for any estimator $\widehat{\mathcal H}=\widehat{\mathcal H}(A)$ we have 
\begin{align*}
    \Pb'\Big( |\widehat{\mathcal H}(A) \cap \mathcal H| \geq \delta\cdot s\tbinom{n}{d} \Big)=o(1) \mbox{ for any constant } \delta>0 \,.
\end{align*}
Our proof relies on a generalization of Fano's inequality which appears in
e.g. \cite{SW12}. The formulation we use here comes from \cite[Lemma~20]{BHM12} (see also \cite[Theorem~3.1]{HM23}). This generalization departs from the classical Fano’s inequality by adopting a broader notion of what constitutes a successful learning algorithm.
\begin{lemma}[Lemma~20 in \cite{BHM12}]
    Let $\mathcal W$ be a hypothesis class of cardinality $M$ with $Z$ being a hypothesis drawn uniformly at random from $\mathcal W$. Let $X$ and $\widehat{Z} \in \mathcal W$ be random variables defined in such a way that $Z \longrightarrow X \longrightarrow \widehat{Z}$ constitutes a Markov chain. We define $d:\mathcal W \times \mathcal W \to \mathbb R_{\geq 0}$ to be a distance function and, for a given $d>0$, we say that a learner is successful if $d(Z,\widehat{Z})\leq d$. For any $h \in \mathcal W$, we further define $B_d(h)=\{ h' \in \mathcal W: d(h,h') \leq d \}$ and we let and we let $M_d=\max_{h \in \mathcal W}\{ |B_d(h)| \}$. Then we have
    \begin{align*}
        \Pb\big(d(Z,\widehat{Z})\geq d\big) \geq 1-\tfrac{I(Z;X)+1}{\log(M/M_d)}
    \end{align*}
    where $I(Z;X)$ is the mutual information between $Z$ and $X$.
\end{lemma}

For our purpose, for any $K=(1+o(1))s\tbinom{n}{d}$ we let $\mathcal W=\{ \mathcal H \subset \tbinom{[n]}{d}: |\mathcal H|=K \}$. In addition, since adding or deleting an arbitrary $o(K)$ elements of $\widehat{\mathcal H}$ does not affect the result, we can assume that $\widehat{\mathcal H}$ has cardinality $K$. Noting that $\mathcal H \longrightarrow A \longrightarrow \widehat{\mathcal H}$ is a Markov chain and one can define a distance metric over $\mathcal W \times \mathcal W$ by taking $d(\mathcal H,\mathcal H')=\tfrac{|\mathcal H \triangle \mathcal H'|}{2s\binom{n}{d}}$. For any constant $0\leq\epsilon<1$, we then have
\begin{align}{\label{eq-bound-recovery-prob}}
    \Pb\big(d(\mathcal H,\widehat{\mathcal H})\leq \epsilon\mid|\mathcal H|=K \big) \leq \tfrac{I(\mathcal H;A)+1}{\log(M/M_\epsilon)} \,,
\end{align}
where $M=|\mathcal W|$ and $M_{\epsilon}=|\{ \mathcal H' \in \mathcal W: |\mathcal H' \cap \mathcal H|\geq (1-\epsilon)s\binom{n}{d} \}|$. Clearly we have (recall that $K=n^{1+\delta+o(1)}$)
\begin{align*}
    &M=\binom{\tbinom{n}{d}}{K}=\exp\big((d-1-\delta+o(1))K\log n\big) \,, \\
    &M_{\epsilon} \leq \binom{\tbinom{n}{d}}{\epsilon K} \cdot \binom{K}{(1-\epsilon)K} \leq \exp\big((d-1-\delta+o(1))\epsilon K\log n\big) \,.
\end{align*}
We then have
\begin{equation}{\label{eq-bound-M/M_eps}}
    \log(M/M_{\epsilon}) \geq (1-\epsilon)(d-1-\delta-o(1))K\log n \,.
\end{equation}
In addition, since $p=o(1)$ and $q=0$, we then have 
\begin{equation}{\label{eq-bound-I(H;A)}}
    I(\mathcal H;A)=H(A)-H(A \mid \mathcal H) \leq H(A) \leq \log\Big( \binom{\binom{n}{2}}{pK\binom{d}{2}} \Big) \leq (1-\delta)pK\log n \,.
\end{equation}
Plugging \eqref{eq-bound-M/M_eps} and \eqref{eq-bound-I(H;A)} into \eqref{eq-bound-recovery-prob}, we get that
\begin{align*}
    \Pb\big(d(\mathcal H,\widehat{\mathcal H})\leq\epsilon\mid|\mathcal H|=K \big) = o(1) \mbox{ for all } K=(1+o(1))s\tbinom{n}{d} \,.
\end{align*}
This verifies Item~(1) of Proposition~\ref{prop-information-recovery-lower-bound}.

To prove Item~(2), without loss of generality, we assume \(p=1\). A crucial observation (as shown in \cite[Definition~1.6 and Lemma~2.5]{BGPY25+}) is that we only need to prove the following result: suppose we sample \((\mathcal H,A)\) according to \(\Pb\) and we sample \(\mathcal H'\) from the posterior distribution of \(\mathcal H\) given \(A\), i.e. from
\[
\mu_A(\mathcal H'):=\mathbb P \big( \mathcal H' \mid A \big).
\]
Then it remains for us to prove that for all fixed \(\iota=\Omega(1)\) we have 
\begin{equation}\label{eq-goal-partial-recovery-lower-bound}
    \Pb\Big(|\mathcal H \cap \mathcal H'|\geq\iota s\tbinom{n}{d}\Big)=o(1) \,.
\end{equation} 
Firstly, for the posterior we have
\begin{align}
	\mu_A(\mathcal H')
	&=\mathbb P\big(\mathcal H' \mid A \big) \propto \mathbb P(\mathcal H')\mathbb P(A\mid \mathcal H') \nonumber\\
	&\propto \frac{\mathbb P(A\mid \mathcal H')}{\mathbb Q (A)}\mathbb P(\mathcal H') =\frac{1}{L(A)}\frac{\mathbb P(A\mid \mathcal H')}{\mathbb Q (A)}\mathbb P(\mathcal H')\label{eq:posterior-representation}
\end{align}
where in the last equality \(L(A)=\frac{\mathbb P(A)}{\mathbb Q(A)}\), and we recall that \(L(A)\) is defined in Section~\ref{sec:proof-prop-1.5}. Thus, for any positive sequence \(\epsilon_n\to 0\) we have
\begin{equation}\label{eq:first-control-LR}
	\mathbb P(L(A)\leq \epsilon_n)=\mathbb E_{\mathbb Q}\left[ L(A)\mathbf 1_{\{L(A)\leq \epsilon_n\}} \right]\leq \epsilon_n.
\end{equation}
Now let
\[
\mathcal B(\mathcal H)=\left\{ \widetilde{\mathcal H}\subset \operatorname U_{n,d}:\, | \mathcal H\cap \widetilde{\mathcal H}|\geq \iota s\binom{n}{d} \right\}.
\]
Then we obtain
\begin{equation}
	\begin{split}
		\mu_A(\mathcal B(\mathcal H))
		&= \mu_A(\mathcal B(\mathcal H))\mathbf 1_{\{L(A)\leq \epsilon_n\}}
		+ \mu_A(\mathcal B(\mathcal H))\mathbf 1_{\{L(A)\geq \epsilon_n\}}\\
		&\leq \mathbf 1_{\{L(A)\leq \epsilon_n\}}
		+ \frac{1}{\epsilon_n}\mathbb E_{\mathcal H'\sim \mu}\left[
		\underbrace{\frac{\mathbb P(A\mid \mathcal H')}{\mathbb Q(A)}}_{=:L_{\mathcal H'}(A)}
		\mathbf 1_{\{\mathcal H'\in \mathcal B(\mathcal H)\}}
		\right]\,,
	\end{split}\notag
\end{equation}
where the inequality follows because we used \(L(A)\geq \epsilon_n\) in the second summand. Then we obtain
\begin{align}
    \mathbb P\left(|\mathcal H'\cap \mathcal H|\geq \iota s\binom nd\right)
	&=\mathbb E_\mathbb P[\mu_A(\mathcal B(\mathcal H))] \nonumber \\
	&\leq \mathbb P(L(A)\leq \epsilon_n)
	+\frac{1}{\epsilon_n} \mathbb E_{A\sim \mathbb P} \mathbb E_{\mathcal H'\sim \mu}\Big[ L_{\mathcal H'}(A)\mathbf 1_{\{\mathcal H'\in \mathcal B(\mathcal H)\}} \Big] \nonumber \\
	&\leq \epsilon_n + \frac{1}{\epsilon_n} \mathbb E_{\mathbb P} \mathbb E_{\mathcal H'\sim \mu}\Big[ L_{\mathcal H'}(A)\mathbf 1_{\{ \mathcal H'\in \mathcal B(\mathcal H)\}} \Big]\,,  \label{eq:expansion-via-LR}
\end{align}
where the last inequality is from \eqref{eq:first-control-LR}. It remains to control the second term. To this end, we have the following replica representation
\begin{equation}\label{eq:replica-repre}
	\begin{split}
		& \mathbb E_{\mathbb P} \mathbb E_{\mathcal H'\sim \mu}\Big[
		L_{\mathcal H'}(A)\mathbf 1_{\{\mathcal H'\in \mathcal B(\mathcal H)\}}
		\Big] \\
		=\;& \mathbb E_{\mathcal H\sim \mu}
		\mathbb E_{A\sim \mathbb P(\cdot\mid \mathcal H)}
		\mathbb E_{\mathcal H'\sim \mu}\Big[
		L_{\mathcal H'}(A)\mathbf 1_{\{\mathcal H'\in \mathcal B(\mathcal H)\}}
		\Big] \\
		=\;& \mathbb E_{\mathcal H,\mathcal H'\overset{iid}{\sim} \mu } \bigg[
		\mathbb E_{A\sim \mathbb Q}\Big[ L_{\mathcal H}(A)L_{\mathcal H'}(A) \Big]
		\mathbf 1_{\{|\mathcal H\cap \mathcal H'|\geq \iota s\binom nd\}}
		\bigg].
	\end{split}
\end{equation}
Following the proof of Lemma~\ref{lem-second-moment-relax}, we obtain (recall that \(\mathcal P(\mathcal H)\) denotes the projection of \(\mathcal H\), and here we take \(\alpha=1\))
\begin{equation}\label{eq:est-replicas-ave-over-A}
	\mathbb E_{A\sim \mathbb Q}\Big[ L_{\mathcal H}(A)L_{\mathcal H'}(A) \Big]
	\leq
	\exp\Big(n^{1-\beta+o(1)}\cdot \big| E(\mathcal P(\mathcal H)) \cap E(\mathcal P(\mathcal H')) \big| \Big).
\end{equation}
Moreover, note that
\[
|\mathcal H\cap\mathcal H'|\geq \iota s\binom nd
\qquad\Longrightarrow\qquad
|E(\mathcal P(\mathcal H))\cap E(\mathcal P(\mathcal H'))|
\geq \iota s\binom nd\binom d2.
\]
Therefore, by \eqref{eq:replica-repre}, it suffices to bound
\begin{equation}
	\mathbb E_{\mathcal H,\mathcal H'\overset{iid}{\sim}\mu }\bigg[
	\exp\Big(n^{1-\beta+o(1)}\cdot \big| E(\mathcal P(\mathcal H)) \cap E(\mathcal P(\mathcal H')) \big| \Big)
	\mathbf 1_{\{ |E(\mathcal P(\mathcal H))\cap E(\mathcal P(\mathcal H'))| \geq \iota s\binom nd\binom d2 \}}
	\bigg].
\end{equation}
Recall that we write
\[
Y=Y(\mathcal H,\mathcal H'):=|E(\mathcal P(\mathcal H))\cap E(\mathcal P(\mathcal H'))|.
\]
By Lemma~\ref{lem-est-tail-intersection}, and using that
\[
\iota s\binom nd\binom d2=\Theta(n^{1+\delta})\gg n^{2\delta},
\]
we obtain
\begin{align}
    &\mathbb E_{\mathcal H,\mathcal H'\overset{iid}{\sim}\mu}\Big[\exp\big(n^{1-\beta+o(1)}Y\big)\mathbf 1_{\{Y\geq \iota s \binom nd\binom d2\}}\Big] \nonumber \\
	\leq{}\;& \int_{\iota s\binom nd\binom d2}^\infty
	n^{1-\beta+o(1)}\exp\big(n^{1-\beta+o(1)}t\big)\mathbb P(Y\geq t)\,\mathrm dt \nonumber \\
	\leq{}\;& \int_{\iota s\binom nd\binom d2}^\infty
	n^{1-\beta+o(1)}\exp\big(n^{1-\beta+o(1)}t\big)\exp\left(-\frac{t^2}{n^{2\delta}}\right)\,\mathrm dt \nonumber \\
	={}\;& \int_{n^{1+o(1)}}^\infty
	n^{1-\beta+\delta+o(1)}\exp\big(n^{1-\beta+\delta+o(1)}s\big)\exp(-s^2)\,\mathrm ds \nonumber \\
	={}\;& \exp\big(n^{2-2\beta+2\delta+o(1)}\big)
	\int_{n^{1+o(1)}}^\infty
	\exp\left(-\left(s-\frac{n^{1-\beta+\delta+o(1)}}{2}\right)^2\right)\,\mathrm ds \,.  \label{eq:final-upper-tail-construction}
\end{align}
Recall that we assume
\[
\delta\leq \frac{d-1}{d+1}
\qquad\text{and}\qquad
\beta > \frac{2\delta}{d(d-1)}+\frac{d-2}{d}.
\]
In particular, this implies
\[
-\beta+\delta=-\Omega(1).
\]
Substituting this into \eqref{eq:final-upper-tail-construction}, we conclude that
\begin{equation*}
	\eqref{eq:final-upper-tail-construction}
	\leq
	\exp\big(n^{2-2\beta+2\delta+o(1)}\big)\exp\big(-n^{2+o(1)}\big)
	=
	\exp\big(-n^{2+o(1)}\big).
\end{equation*}
As a consequence,
\[
\mathbb E_\mathbb P\mathbb E_{\mathcal H'\sim \mu}\Big[L_{\mathcal H'}(A)\mathbf 1_{\{\mathcal H'\in \mathcal B(\mathcal H)\}}\Big]
\leq
\exp\big(-n^{2+o(1)}\big).
\]
Combining this with \eqref{eq:expansion-via-LR}, and choosing \(\epsilon_n=o(1)\) appropriately, yields the desired result \eqref{eq-goal-partial-recovery-lower-bound}.

\section{Proof of Proposition~\ref{prop-information-almost-exact-recovery-lower-bound}}{\label{sec:proof-prop-1.9}}

In this section, without loss of generality, we assume $q=0$. Following the approach in \cite{DWXY23}, we only need to prove the following result: suppose we sample $(\mathcal H,A)$ according to $\Pb$ and we sample $\mathcal H'$ according to the posterior distribution of $\mathcal H$ given $A$. Then we have
\begin{equation}{\label{eq-goal-almost-exact-recovery-lower-bound}}
    \Pb\Big(|\mathcal H \triangle \mathcal H'|\leq o(1)\cdot s\tbinom{n}{d}\Big)=o(1) \,.
\end{equation}
To see this, note that for any estimator $\widehat{\mathcal H}=\widehat{\mathcal H}(A)$, we have that $(\widehat{\mathcal H},\mathcal H)$ and $(\widehat{\mathcal H},\mathcal H')$ are equal in law. Combining the fact that 
\begin{align*}
    |\mathcal H \triangle \widehat{\mathcal H}| + |\mathcal H' \triangle \widehat{\mathcal H}| \geq |\mathcal H \triangle \mathcal H'| \,,
\end{align*}
we then have
\begin{align*}
    \Omega(1) \cdot s\tbinom{n}{d} \leq \mathbb E_{\Pb}\big[ |\mathcal H \triangle \mathcal H'| \big] \leq \mathbb E\big[ |\mathcal H \triangle \widehat{\mathcal H}| + |\mathcal H' \triangle \widehat{\mathcal H}| \big] = 2\mathbb E_{\Pb}\big[ |\mathcal H \triangle \widehat{\mathcal H}| \big] \,,
\end{align*}
where the first inequality follows from \eqref{eq-goal-almost-exact-recovery-lower-bound}. This yields Proposition~\ref{prop-information-almost-exact-recovery-lower-bound}. 

We now prove \eqref{eq-goal-almost-exact-recovery-lower-bound}. For any $\mathcal H \in \{ 0,1 \}^{\binom{[n]}{d}}$ and $A \in \{ 0,1 \}^{\binom{[n]}{2}}$, define
\begin{equation}{\label{eq-def-empty-H;A}}
    \mathsf{Empty}(\mathcal H;A)\triangleq\Big\{ \Psi\in\mathcal H: E(\mathcal P(\Psi)) \cap E(A)=\emptyset \Big\} \,.
\end{equation}
Define
\begin{equation}{\label{eq-def-mathcal-G}}
\begin{aligned}
    \mathcal G:=\Big\{ & (\mathcal H,A):|\mathcal H|=(1+o(1))s\tbinom{n}{d}, \mathsf{Empty}(\mathcal H;A) \geq \Omega(1) \cdot s\tbinom{n}{d} ; \\
    & |E(\mathcal P(\mathsf{Empty}(\mathcal H;A)))| \geq (1-o(1)) \tbinom{d}{2} |\mathsf{Empty}(\mathcal H;A)| \Big\} \,.
\end{aligned}
\end{equation}
It is easy to check that $\Pb((\mathcal H,A)\in\mathcal G)=1-o(1)$. In addition, for any $(\mathcal H,A)\in\mathcal G$, denote 
\begin{equation}{\label{eq-def-blur(H;A)}}
    \begin{aligned}
        \mathsf{Blur}(\mathcal H;A)\triangleq\Big\{ &\mathcal H' \in\{0,1\}^{\binom{[n]}{d}}: |\mathcal H'|=|\mathcal H|, \mathcal H \setminus \mathsf{Empty}(\mathcal H;A) \subset \mathcal H' \Big\} \,.
    \end{aligned}
\end{equation}
It is straightforward to check that for $(\mathcal H,A)\in\mathcal G$
\begin{align*}
    \#\mathsf{Blur}(\mathcal H;A) \geq \binom{ \binom{n}{d}-(1+o(1))s\binom{n}{d} }{ |\mathsf{Empty}(\mathcal H;A)| } \geq \exp\Big( \Omega(1)\cdot s\tbinom{n}{d}\log n \Big) \,.
\end{align*}
In addition, for all $\mathcal H'\in\mathsf{Blur}(\mathcal H;A)$ we have
\begin{align*}
    \frac{ \Pb(\mathcal H\mid A) }{ \Pb(\mathcal H'\mid A) } &= \frac{ \Pb(\mathcal H,A) }{ \Pb(\mathcal H',A) } = \frac{ \mu(\mathcal H) \Pb_{\mathcal H}(A) }{ \mu(\mathcal H') \Pb_{\mathcal H'}(A) } = \frac{ \Pb_{\mathcal H}(A) }{ \Pb_{\mathcal H'}(A) } \\
    &= \frac{ p^{|E(A)|} (1-p)^{ |E(\mathcal P(\mathcal H)) \setminus E(A)| } }{ p^{|E(A)|} (1-p)^{ |E(\mathcal P(\mathcal H')) \setminus E(A)| } } \leq \exp\Big( o(1)\cdot s\tbinom{n}{d} \Big) \,,
\end{align*}
where the third equality follows from $|\mathcal H|=|\mathcal H'|$ and the inequality follows from 
\begin{align*}
    &|E(\mathcal P(\mathcal H)) \setminus E(A)|-|E(\mathcal P(\mathcal H')) \setminus E(A)| \\
    =\ &|E(\mathcal{P}(\mathsf{Empty}(\mathcal H;A)))|-|E(\mathcal P(\mathcal H' \cap \mathsf{Empty}(\mathcal H;A))) \setminus E(A)| \\
    \geq\ &(1-o(1))\tbinom{d}{2}|\mathsf{Empty}(\mathcal H;A)| - \tbinom{d}{2}|\mathsf{Empty}(\mathcal H;A)| = - o(1)\cdot s\tbinom{n}{d} \,.
\end{align*}
Thus, we see that for $(\mathcal H,A)\in\mathcal G$
\begin{equation}{\label{eq-est-Pb(H|A)}}
    \Pb(\mathcal H\mid A) \leq \frac{ \exp(o(1)\cdot s\tbinom{n}{d})  }{ |\mathsf{Blur}(\mathcal H;A)| } = \exp\Big( -\Omega(1)\cdot s\tbinom{n}{d} \log n \Big) \,.
\end{equation}
Now we return to the proof of \eqref{eq-goal-almost-exact-recovery-lower-bound}. Clearly we can write the left hand side of \eqref{eq-goal-almost-exact-recovery-lower-bound} as
\begin{align}
    & \sum_{ \substack{ \mathcal H,\mathcal H',A \\ |\mathcal H \triangle \mathcal H'|\leq o(1)\cdot s\binom{n}{d} } } \frac{ \Pb(\mathcal H,A)\Pb(\mathcal H',A) }{ \Pb(A) } \nonumber \\
    \leq\ & \sum_{ \substack{ \mathcal H,\mathcal H';A:(\mathcal H,A), (\mathcal H',A) \in\mathcal G \\ |\mathcal H \triangle \mathcal H'|\leq o(1)\cdot s\binom{n}{d} } } \frac{ \Pb(\mathcal H,A)\Pb(\mathcal H',A) }{ \Pb(A) } + 2\sum_{ \substack{ \mathcal H,\mathcal H';A \\ (\mathcal H,A)\not\in\mathcal G } } \frac{ \Pb(\mathcal H,A)\Pb(\mathcal H',A) }{ \Pb(A) } \label{eq-almost-recovery-relax-1}
\end{align}
Note that
\begin{align}
    \sum_{ \substack{ \mathcal H,\mathcal H';A \\ (\mathcal H,A)\not\in\mathcal G } } \frac{ \Pb(\mathcal H,A)\Pb(\mathcal H',A) }{ \Pb(A) } = \sum_{ \substack{ (\mathcal H,A)\not\in\mathcal G } } \Pb(\mathcal H,A) \sum_{\mathcal H'} \Pb(\mathcal H'\mid A) \leq \Pb((\mathcal H,A)\not\in\mathcal G)=o(1) \,. \label{eq-almost-recovery-relax-2}
\end{align}
In addition, we have
\begin{align}
    &\sum_{ \substack{ \mathcal H,\mathcal H';A:(\mathcal H,A), (\mathcal H',A) \in\mathcal G \\ |\mathcal H \triangle \mathcal H'|\leq o(1)\cdot s\binom{n}{d} } } \frac{ \Pb(\mathcal H,A)\Pb(\mathcal H',A) }{ \Pb(A) } = \sum_{ \substack{ \mathcal H,\mathcal H';A:(\mathcal H,A), (\mathcal H',A) \in\mathcal G \\ |\mathcal H \triangle \mathcal H'|\leq o(1)\cdot s\binom{n}{d} } } \Pb(\mathcal H',A) \Pb(\mathcal H\mid A) \nonumber \\
    \overset{\eqref{eq-est-Pb(H|A)}}{\leq}\ & \exp\Big( -\Omega(1)\cdot s\tbinom{n}{d} \log n \Big) \cdot \sum_{ \substack{ \mathcal H,\mathcal H';A:(\mathcal H,A),(\mathcal H',A) \in\mathcal G \\ |\mathcal H \triangle \mathcal H'|\leq o(1)\cdot s\binom{n}{d} } } \Pb(\mathcal H',A) \nonumber \\
    \leq\ & \exp\Big( -\Omega(1)\cdot s\tbinom{n}{d} \log n \Big) \cdot \sum_{ \substack{ \mathcal H';A \\ (\mathcal H',A) \in \mathcal G } } \Pb(\mathcal H',A) \#\Big\{ \mathcal H: |\mathcal H\triangle \mathcal H'|\leq o(1)\cdot s\tbinom{n}{d} \Big\} \nonumber \\
    \leq\ & \exp\Big( -\Omega(1)\cdot s\tbinom{n}{d} \log n \Big) \cdot \exp\Big( o(1)\cdot s\tbinom{n}{d} \log n \Big)  = o(1) \,, \label{eq-almost-recovery-relax-3}
\end{align}
where the third inequality follows from for $|\mathcal H'|=(1+o(1))s\tbinom{n}{d}$
\begin{align*}
    \#\Big\{ \mathcal H: |\mathcal H\triangle \mathcal H'|\leq o(1)\cdot s\tbinom{n}{d} \Big\} \leq 2^{|\mathcal H'|} \cdot \binom{ \binom{n}{d} }{ o(1)\cdot s\binom{n}{d} } = \exp\Big( o(1)\cdot s\tbinom{n}{d} \log n \Big) \,.
\end{align*}
And the result follows directly by plugging \eqref{eq-almost-recovery-relax-2} and \eqref{eq-almost-recovery-relax-3} into \eqref{eq-almost-recovery-relax-1}.

\section*{Acknowledgment}

The main part of the work was carried out when Q.~Xu was at Peking University. S.~Gong and Z.~Li are partially supported by National Key R$\&$D Program of China  (Project No. 2023YFA1010103) and NSFC Key Program (Project No.12231002).

\appendix

\section{Detection with a modified null distribution}{\label{sec:alternative-detection-prob}}

In the main text, we have cited the presence of a detection-recovery gap as a source of difficulty for obtaining tight thresholds for recovery. One might wonder whether this can be remedied simply by choosing a better null distribution $\Qb$ that is harder to distinguish from the planted distribution $\Pb$. In this section, we will argue that this approach does not work. Recall Definition~\ref{def-noisy-projection}. Define $\widetilde{\Qb}=\widetilde{\Qb}_{n}$ to be the \ER graph with edge-density 
\begin{equation}{\label{eq-def-widetilde-q}}
    \widetilde{q} = \tfrac{1}{\binom{n}{2}} \Big( (p-q)s\tbinom{n}{d}\tbinom{d}{2}  \Big) + q \,.
\end{equation}
Clearly $\widetilde{\Qb}$ and $\Pb$ have the same expectational edge-density. However, we will show that it is still possible to distinguish $\Pb$ from $\widetilde{\Qb}$ when $\delta,\beta<\tfrac{d-1}{d+1}$. Consider the ``clique counting'' statistic 
\begin{equation}{\label{eq-def-statistics-g}}
    g(A)=\sum_{\Psi\in\operatorname{U}}\mathbf 1_{ \{\mathcal P(\Psi) \subset A\} } \,.
\end{equation}
\begin{proposition}
    Suppose $\delta,\beta<\tfrac{d-1}{d+1}$. We then have
    \begin{align*}
        &\Pb\Big( g(A) \geq q^{\binom{d}{2}}\tbinom{n}{d} + \tfrac{1}{2} ps\tbinom{n}{d} \Big) = 1-o(1) \,, \\
        &\widetilde{\Qb}\Big( g(A) < q^{\binom{d}{2}}\tbinom{n}{d} + \tfrac{1}{2} ps\tbinom{n}{d} \Big) = 1-o(1) \,.
    \end{align*}
\end{proposition}
\begin{proof}
    Denote $\kappa=\max\{\delta,\beta\}$ and recall from \eqref{eq-def-widetilde-q} that $\widetilde{q}=n^{-1+\kappa+o(1)}$. When $\kappa<\tfrac{d-3}{d-1}$, it can be easily checked using Markov's inequality that $\widetilde{\Qb}(g(A)=0)=1-o(1)$. When $\kappa \geq \tfrac{d-3}{d-1}$, we have
    \begin{align*}
        \operatorname{Var}_{\widetilde{\Qb}}(g(A))= \sum_{\Psi,\Phi\in\operatorname{U}} \operatorname{Cov}_{\widetilde{\Qb}}\Big( \mathbf 1_{ \{\mathcal P(\Psi) \subset A\} }, \mathbf 1_{ \{\mathcal P(\Phi) \subset A\} } \Big) \,.
    \end{align*}
    Note that for $2\leq m \leq d$, we have $\#\{ \Psi,\Phi: |V(\Psi)\cap V(\Phi)|=m \}=n^{2d-m+o(1)}$ and for $|V(\Psi)\cap V(\Phi)|=m$ we have
    \begin{align*}
        \operatorname{Cov}_{\widetilde{\Qb}}\big( \mathbf 1_{ \{\mathcal P(\Psi) \subset A\} }, \mathbf 1_{ \{\mathcal P(\Phi) \subset A\} } \big) \leq \widetilde{q}^{ 2\binom{d}{2}-\binom{m}{2} } \,.
    \end{align*}
    Thus, we have
    \begin{align*}
        \operatorname{Var}_{\widetilde{\Qb}}(g(A))= \mathbb E_{\widetilde{\Qb}}[g(A)]^2 \cdot \sum_{2 \leq m \leq d} n^{ -m+(1-\kappa)\binom{m}{2}+o(1) } \,.
    \end{align*}
    Note that when $\kappa>\tfrac{d-3}{d-1}$, we have the maximum of $-m+(1-\kappa)\binom{m}{2}$ where $2 \leq m \leq d$ is taken by either $m=2$ or $m=d$. Thus, 
    \begin{align*}
        \operatorname{Var}_{\widetilde{\Qb}}(g(A))= \mathbb E_{\widetilde{\Qb}}[g(A)]^2 \cdot  n^{-\min\{(1+\kappa),d-\binom{d}{2}(1-\kappa)\}+o(1)} \,.
    \end{align*}
    This yields that
    \begin{align*}
        \operatorname{Var}_{\widetilde{\Qb}}(g(A))^{\tfrac{1}{2}} \leq n^{ d-\binom{d}{2}(1-\kappa)-\tfrac{1}{2}\min\{(1+\kappa),d-\binom{d}{2}(1-\kappa)\}+o(1) } \leq n^{1+\delta-\Omega(1)} \,,
    \end{align*}
    where the last inequality follows from our assumption that (recall $\kappa=\max\{ \delta,\beta \}$)
    \begin{align*}
        \tfrac{d-\tbinom{d}{2}(1-\kappa)}{2}, \  d-\tbinom{d}{2}(1-\kappa)-\tfrac{1+\kappa}{2} = 1-\Omega(1) \mbox{ when } \beta,\delta<\tfrac{d-1}{d+1} \,.
    \end{align*}
    In addition, we have 
    \begin{align*}
        \mathbb E_{\widetilde{\Qb}}[g(A)] = \widetilde{q}^{\binom{d}{2}}\tbinom{n}{d} = q^{\binom{d}{2}}\tbinom{n}{d} + O(1) \cdot \tbinom{n}{d} \sum_{k=0}^{\binom{d}{2}-1} q^{k} (\widetilde{q}-q)^{\binom{d}{2}-k} 
    \end{align*}
    Note that $\widetilde{q}-q=n^{-1+\delta+o(1)}$ and $q=n^{-1+\beta+o(1)}$, thus we have for all $0 \leq k \leq \tbinom{d}{2}-1$
    \begin{align*}
        \tbinom{n}{d} q^{k} (\widetilde{q}-q)^{\binom{d}{2}-k} &= n^{ d-k(1-\beta)-(\binom{d}{2}-k)(1-\delta)+o(1) } \\
        &\leq n^{d-(\tbinom{d}{2}-1)(1-\kappa)-(1-\delta)+o(1)} \ll ps\tbinom{n}{d} = n^{(1+\delta)+o(1)} \,,
    \end{align*}
    where the last inequality follows from the fact that 
    \begin{align*}
        d-(\tbinom{d}{2}-1)(1-\kappa)-(1-\delta)<(1+\delta) \mbox{ when } \delta,\beta<\tfrac{d-1}{d+1}  \,.
    \end{align*}
    In conclusion, we have shown that 
    \begin{align*}
        \mathbb E_{\widetilde{\Qb}}[g(A)] = q^{\binom{d}{2}}\tbinom{n}{d} + o(1) \cdot ps\tbinom{n}{d} \mbox{ and } \operatorname{Var}_{\widetilde{\Qb}}[g(A)]^{\frac{1}{2}}= o(1)\cdot ps\tbinom{n}{d} \,.
    \end{align*}
    Thus from standard Chebyshev's inequality we have
    \begin{align*}
        \widetilde{\Qb}\Big( g(A) < q^{\binom{d}{2}}\tbinom{n}{d} + \tfrac{1}{2} ps\tbinom{n}{d} \Big) = 1-o(1) \,.
    \end{align*}
    In addition, note that under $\Pb$ we have $|\mathcal H|\geq 0.9ps\tbinom{n}{d} \gg q^{\binom{d}{2}} \tbinom{n}{d}$ with probability $1-o(1)$ when $\beta<\tfrac{d-1}{d+1}$. Thus, it is clear that
    \begin{equation*}
        \Pb\Big( g(A) \geq q^{\binom{d}{2}}\tbinom{n}{d} + \tfrac{1}{2} ps\tbinom{n}{d} \Big) = 1-o(1) \,. \qedhere
    \end{equation*}
\end{proof}

\small
\bibliographystyle{alpha}
\bibliography{ref}

@inproceedings{BGPY25+,
    author = {Guy Bresler and Chenghao Guo and Yury Polyanskiy and Andrew Yao},
    title = {Partial and Exact Recovery of a Random Hypergraph from its Graph Projection},
    booktitle = {Proceedings of 38th Conference on Learning Theory (COLT)},
    pages= {534--593},
    publisher={PMLR},
    year = {2025}
}

@inproceedings{BGP24,
    author = {Guy Bresler and Chenghao Guo and Yury Polyanskiy},
    title = {Thresholds for Reconstruction of Random Hypergraphs From Graph Projections},
    booktitle = {Proceedings of the 37th Annual Conference on Learning Theory (COLT)},
    pages = {632--647},
    publisher = {PMLR},
    year = {2024}
}

@book{AS16,
    author = {Noga Alon and Joel H. Spencer},
    title = {The probabilistic method},
    publisher = {Wiley Publishing},
    year = {2016}
}

@article{SW12,
    author = {Narayana P. Santhanam and Martin J. Wainwright},
    title = {Information-theoretic limits of selecting binary graphical models in high dimensions},
    journal = {IEEE Transactions on Information Theory},
    volume = {58},
    issue = {7},
    pages = {4117--4134},
    year = {2012}
}

@inproceedings{BHM12,
    author = {Siddhartha Banerjee and Nidhi Hegde and Laurent Massouli\'e},
    title = {The price of privacy in untrusted recommendation engines},
    booktitle = {Proceedings of the 50th Annual Allerton Conference on Communication, Control, and Computing (Allerton)},
    pages = {920--927},
    publisher = {IEEE},
    year = {2012}
}

@article{HM23,
    author = {Georgina Hall and Laurent Massouli\'e},
    title = {Partial Recovery in the Graph Alignment Problem},
    journal = {Operations Research},
    volume = {71},
    pages = {259--272},
    year = {2023} 
}

@article{DWXY23,
    author = {Jian Ding and Yihong Wu and Jiaming Xu and Dana Yang},
    title = {The planted matching problem: Sharp threshold and infinite-order phase transition},
    journal = {Probability Theory and Related Fields},
    volume = {187},
    pages = {1--71},
    year = {2023}
}

@article{GD14,
    author = {Debarghya Ghoshdastidar and Ambedkar Dukkipati},
    title = {Consistency of spectral partitioning of uniform hypergraphs under planted partition model},
    journal = {Annals of Statistics},
    volume = {45},
    pages = {289--315},
    year = {2017}
}

@inproceedings{ACKZ15,
    author = {Maria C. Angelini and Francesco Caltagirone and Florent Krzakala and Lenka Zdeborov\'a},
    title = {Spectral detection on sparse hypergraphs},
    booktitle = {Proceedings of the 53rd Annual Allerton Conference on Communication, Control, and Computing (Allerton)},
    pages = {66--73},
    publisher = {IEEE},
    year = {2015}
}

@article{New04,
    author = {Mark E. J. Newman},
    title = {Coauthorship networks and patterns of scientific collaboration},
    journal = {Proceedings of the National Academy of Sciences of the United States of America},
    volume = {101},
    pages = {5200--5205},
    year = {2004}
}

@article{New06,
    author = {Mark E. J. Newman},
    title = {Modularity and community structure in networks},
    journal = {Proceedings of the National Academy of Sciences of the United States of America},
    volume = {103},
    pages = {8577--8582},
    year = {2006}
}

@inproceedings{WK24,
    author = {Yanbang Wang and Jon Kleinberg},
    title = {From graphs to hypergraphs: Hypergraph projection and its reconstruction},
    booktitle = {Proceedings of the 12th International Conference on Learning Representations (ICLR)},
    publisher = {},
    year = {2024}
}

@article{YPP21,
    author = {Jean-Gabriel Young and Giovanni Petri and Tiago P. Peixoto},
    title = {Hypergraph reconstruction from network data},
    journal = {Communications Physics},
    volume = {4},
    year = {2021}
}

@article{LYA23,
    author = {Simon Lizotte and Jean-Gabriel Young and Antoine Allard},
    title = {Hypergraph reconstruction from uncertain pairwise observations},
    journal = {Scientific Reports},
    volume = {13},
    year = {2023}
}

@inproceedings{CFLL24,
    author = {Yang Chen and Cong Fang and Zhouchen Lin and Bing Liu},
    title = {Relational learning in pretrained models: A theory from hypergraph recovery perspective},
    booktitle = {Proceedings of the 41st International Conference on Machine Learning (ICML)},
    pages = {6666--6698},
    publisher = {PMLR},
    year = {2024}
}

@inproceedings{GP24,
    author = {Yuzhou Gu and Yury Polyansky},
    title = {Community detection in the hypergraph stochastic block model and reconstruction on hypertrees},
    booktitle = {Proceedings of the 37th Annual Conference on Learning Theory (COLT)},
    pages = {2166--2203},
    publisher = {PMLR},
    year = {2024}
}

@inproceedings{GP23,
    author = {Yuzhou Gu and Yury Polyansky},
    title = {Weak recovery threshold for the hypergraph stochastic block model},
    booktitle = {Proceedings of the 36th Annual Conference on Learning Theory (COLT)},
    pages = {885--920},
    publisher = {PMLR},
    year = {2023}
}

@inproceedings{GJ23,
    author = {Julia Gaudio and Nirmit Joshi},
    title = {Community detection in the hypergraph {SBM}: Exact recovery given the similarity matrix},
    booktitle = {Proceedings of 36th Conference on Learning Theory (COLT)},
    pages = {469--510},
    publisher = {PMLR},
    year = {2023}
}

@article{DW23,
    author = {Dumitriu, Ioana and Wang, Hai-Xiao},
    title = {Optimal and exact recovery on general non-uniform Hypergraph Stochastic Block Model},
    journal = {Annals of Statistics},
    volume = {54},
    issue = {1},
    pages = {48--73},
    year = {2026}
}

@inproceedings{SZ22,
    author = {Ludovic Stephan and Yizhe Zhu},
    title = {Sparse random hypergraphs: Non-backtracking spectra and community detection},
    booktitle = {IEEE 63rd Annual Symposium on Foundations of Computer Science (FOCS)},
    pages = {567--575},
    publisher = {IEEE},
    year = {2022}
}

@inproceedings{KY04,
    author = {Klimt, Bryan and Yang, Yiming},
    title = {Introducing the enron corpus},
    booktitle = {International Conference on Email and Anti-Spam},
    pages = {92--96},
    year = {2004}
}

@article{DJ04,
    author = {Donoho, David and Jin, Jiashun},
    title = {Higher criticism for detecting sparse heterogeneous mixtures},
    journal = {Annals of Statistics},
    volume = {32},
    pages = {962--994},
    year = {2004}
}

@inproceedings{MST19,
    author = {Massouli\'e, Laurent and Stephan, Ludovic and Towsley, Don},
    title = {Planting trees in graphs, and finding them back},
    booktitle = {Proceedings of the 32th Annual Conference on Learning Theory (COLT)},
    publisher = {PMLR},
    pages = {2341--2371},
    year = {2019}
}

\end{document}